\documentclass[12pt,a4paper,reqno,twoside]{amsart}

\usepackage[english]{babel}
\usepackage[symbol*,ragged]{footmisc}
\usepackage{geometry}
\usepackage{amsmath}
\usepackage{amssymb}
\usepackage{mathrsfs}
\usepackage{float}
\usepackage{booktabs}
\usepackage{xcolor}
\usepackage[colorlinks,hypertexnames=false,linkcolor=red,anchorcolor=blue,citecolor=blue,urlcolor=blue]{hyperref}

\allowdisplaybreaks[4]
\renewcommand{\footnoterule}{
  \kern -3pt
  \hrule width 2.5in height 0.4pt
  \kern 3pt
}

\theoremstyle{plain}
\newtheorem{theorem}{\normalfont\scshape Theorem}[section]
\newtheorem{lemma}[theorem]{\normalfont\scshape Lemma}

\theoremstyle{remark}

\newtheorem*{remark*}{\normalfont\scshape Remark}

\numberwithin{equation}{section}
\addtocounter{footnote}{1}

\makeatletter
\@ifundefined{MakeUppercase}{}{}
\makeatother

\begin{document}

\title[On Waring--Goldbach problem for two squares and higher powers]
      {On Waring--Goldbach problem for two squares and higher powers}

\author[Linji Long, Yunchao Chi, Jinjiang Li, Min Zhang]
       {Linji Long \quad \& \quad Yunchao Chi \quad \& \quad Jinjiang Li \quad \& \quad Min Zhang}

\address{[Linji Long] Department of Mathematics, China University of Mining and Technology,
         Beijing 100083, People's Republic of China}

\email{\textcolor{black}{linji.long.math@gmail.com}}

\address{[Yunchao Chi] Department of Mathematics, China University of Mining and Technology,
	Beijing 100083, People's Republic of China}

\email{\textcolor{black}{yunchao.chi.math@gmail.com}}

\address{[Jinjiang Li] (Corresponding author) Department of Mathematics, China University of Mining and Technology,
         Beijing 100083, People's Republic of China}

\email{\textcolor{black}{jinjiang.li.math@gmail.com}}

\address{[Min Zhang] College of Mathematical Science, Beijing Information Science and Technology University,
		 Beijing 100192, People's Republic of China  }

\email{\textcolor{black}{min.zhang.math@gmail.com}}

\date{}

\footnotetext[1]{Jinjiang Li is the corresponding author. \\
  \quad\,\,
  \textbf{Keywords}: Waring--Goldbach problem; circle method; additive representations\\

  \quad\,\,
  \textbf{MR(2020) Subject Classification}: 11P05, 11P32, 11P55
}

\begin{abstract}
Let $\nu_1(n)$ and $\nu_2(n)$, respectively, denote the number of representations of the natural number $n$ as the following form
\begin{equation*}
	n=p_1^2+p_2^2+p_3^4+p_4^4+p_5^4+p_6^5+p_7^8,
\end{equation*}
and
\begin{equation*}
	n=p_1^2+p_2^2+p_3^3+p_4^4+p_5^5+p_6^6+p_7^8,
\end{equation*}
where $p_1,\dots,p_7$ are primes in both above cases.
In this paper, it is proved that $\nu_1(n)\gg n^{43/40}(\log n)^{-7}$ holds for every sufficiently large integer $n$ subject to $n\equiv1\pmod2$ and $n\equiv\pm1\pmod3$. Moreover, we also establish that
$\nu_2(n)\gg n^{43/40}(\log n)^{-7}$ holds for every sufficiently large integer $n$ subject to $n\equiv1\pmod2$.
These results improve and enhance the earlier work of Andrade \cite{Andrade2026}  and Wang--Tian \cite{WangTian2026}, respectively.
\end{abstract}
\maketitle

\section{Introduction}

Motivated by the general philosophy underlying the Hardy--Littlewood method, it is
conjectured that whenever $s$ and $k$ are positive integers with
$s\geqslant k+1$, every sufficiently large natural number $n$ satisfying the necessary
congruence conditions can be represented as the sum of $s$ $k$--th powers of primes.
In 1938, Hua \cite{Hua1938} proved that every sufficiently large integer
$n\equiv 5\pmod{24}$ is a sum of five squares of primes, while every sufficiently large
integer $n\equiv 17\pmod{240}$ is a sum of seventeen biquadrates of primes. This
result was subsequently improved by Davenport \cite{Davenport1939}, Kawada and
Wooley \cite{KawadaWooley2001}, and Zhao \cite{Zhao2014}, successively and respectively.

In 1981, Hooley \cite{Hooley1981} introduced divisor sum techniques into Waring's
problem with mixed powers and obtained an asymptotic formula for the number of
representations of every sufficiently large integer $n$ in the form
\begin{equation*}
n=x_1^2+x_2^2+x_3^4+x_4^4+x_5^4+x_6^4+x_7^4,
\end{equation*}
with all variables positive integers. Moreover, in the same literature, Hooley \cite{Hooley1981} also investigated
the solvability of the following Diophantine equation
\begin{equation*}
n=x_1^2+x_2^2+x_3^3+x_4^4+x_5^5+x_6^6+x_7^7,
\end{equation*}
with $x_1,x_2,\dots,x_7$ being non--negative integers, and derived the corresponding asymptotic formula.

Motivated by Hooley's results as above, it is reasonable to
conjecture that every sufficiently large integer $n$ subject to some necessary local conditions can be represented as
\begin{equation}\label{eq:conjecture-2244444}
	n=p_1^2+p_2^2+p_3^4+p_4^4+p_5^4+p_6^4+p_7^4,
\end{equation}
and
\begin{equation}\label{eq:conjecture-2234567}
	n=p_1^2+p_2^2+p_3^3+p_4^4+p_5^5+p_6^6+p_7^7,
\end{equation}
where the variables $p_1,\dots,p_7$ in both (\ref{eq:conjecture-2244444}) and (\ref{eq:conjecture-2234567}) are primes. In 2016, Cai \cite{Cai2016} obtained an almost--prime type approximation to the conjecture (\ref{eq:conjecture-2244444}) and proved that every sufficiently large integer
$n\equiv 7$ or $103\pmod{240}$ can be written as
\begin{equation*}
n=x^2+p_2^2+p_3^4+p_4^4+p_5^4+p_6^4+p_7^4,
\end{equation*}
where $p_2,\ldots,p_7$ are primes and $x$ is an almost--prime $\mathcal{P}_4$. Here we denote by $\mathcal{P}_r$
an almost--prime with at most $r$ prime factors, counted according to multiplicity. In addition, in 2015, Cai and Mu \cite{Cai-Mu-2015} established an approximation to the conjecture (\ref{eq:conjecture-2234567}) and showed that, for $6\leqslant k\leqslant19$, every sufficiently large odd integer $n$ can be written as
\begin{equation*}
n=x^2+p_2^2+p_3^3+p_4^4+p_5^5+p_6^6+p_7^k,
\end{equation*}
where $p_2,\dots,p_7$ are primes and $x$ is an almost--prime $\mathcal{P}_{r(k)}$ with explicit value $r(k)$. Especially,
Cai and Mu \cite{Cai-Mu-2015} obtained that $r(6)=6,r(7)=6,r(8)=7$.
Recently, in 2024, Li--Long--Zhang \cite{Li-Long-Zhang-2024} subsequently accomplished the proof of the conjecture (\ref{eq:conjecture-2244444}) by combining the Hardy--Littlewood circle method with iterative exceptional set estimates. Moreover, in 2023, Br\"udern \cite{Brudern-2023} thoroughly solved the conjecture (\ref{eq:conjecture-2234567}).

In 1987, Br\"udern \cite{Brudern1987I,Brudern1987II} enhanced Hooley's work
\cite{Hooley1981} to sparser mixed power representations. More precisely, he proved that every sufficiently large integer $n$ has a
representation in non-negative integers of the form
\begin{equation*}
n=x_1^2+x_2^2+x_3^4+x_4^4+x_5^4+x_6^5+x_7^6.
\end{equation*}
In 2026, Andrade \cite{Andrade2026} established the corresponding result with all the variables being primes. To be specific, Andrade \cite{Andrade2026} proved that every sufficiently large odd integer $n$ with $3\nmid n$ can be represented as
\begin{equation*}
n=p_1^2+p_2^2+p_3^4+p_4^4+p_5^4+p_6^5+p_7^6,
\end{equation*}
where $p_1,\ldots,p_7$ are primes. Later, Long--Li--Zhang \cite{Long-Li-Zhang-2026} enhance the
conclusion of Andrade \cite{Andrade2026} to sparser mixed power and deduce the asymptotic formula of the representation of $n$ in the form
\begin{equation*}
n=p_1^2+p_2^2+p_3^4+p_4^4+p_5^4+p_6^5+p_7^7,
\end{equation*}
provided that $n$ is sufficiently large subject to $n\equiv1\!\pmod2$. On the other hand, for the conjecture
(\ref{eq:conjecture-2234567}), Wang and Tian \cite{WangTian2026} show that the number of representations, denoted by $\nu(n)$, of the sufficiently large odd integer $n$ in the form
\begin{equation*}
n=p_1^2+p_2^2+p_3^3+p_4^4+p_5^5+p_6^6+p_7^8
\end{equation*}
satisfies $\nu(n)\gg n^{39/37}(\log n)^{-7}$.

In the present paper, we shall further improve and enhance the results of Andrade \cite{Andrade2026} and Long--Li--Zhang \cite{Long-Li-Zhang-2026} by replacing the sextic or septic prime variable with an octic prime
variable. Moreover, we also prove that conjecture (\ref{eq:conjecture-2234567}) is true by establishing the lower bound of the representation to the expected order of magnitude.
Our results are as follows.

\begin{theorem}\label{thm:waring-goldbach-2244458}
Let $\nu_1(n)$ denote the number of representations of the natural number $n$ as the following form
\begin{equation}\label{eq:target-representation}
n=p_1^2+p_2^2+p_3^4+p_4^4+p_5^4+p_6^5+p_7^8,
\end{equation}
where $p_1,\dots,p_7$ are primes. Then, for any sufficiently large integer $n$ subject to $n\equiv1\pmod2$
and $n\equiv\pm1\pmod3$, there holds $\nu_1(n)\gg n^{43/40}(\log n)^{-7}$.
\end{theorem}

\begin{theorem}\label{thm:waring-goldbach-2234568}
Let $\nu_2(n)$ denote the number of representations of the natural number n as the following form	
\begin{equation}\label{eq:target-representation-2234568}
		n=p_1^2+p_2^2+p_3^3+p_4^4+p_5^5+p_6^6+p_7^8,
\end{equation}
where $p_1,\dots,p_7$ are primes. Then, for any sufficiently large integer $n$ subject to $n\equiv1\pmod2$, there holds $\nu_2(n)\gg n^{43/40}(\log n)^{-7}$.
\end{theorem}

\begin{remark*}
In order to compare our result with the previous result, we emphasize that
Theorem \ref{thm:waring-goldbach-2234568} establishes the lower bound $\nu_2(n)\gg n^{43/40}(\log n)^{-7}$ which
is suggested by Hardy--Littlewood heuristic, while the result of Wang and Tian \cite{WangTian2026} only gives a weak estimation for the number of representations $\nu_2(n)\gg n^{39/37}(\log n)^{-7}$.
\end{remark*}

\section{Notation and Preliminaries for Theorem \ref{thm:waring-goldbach-2244458}}
\label{sec:notation}

Throughout this paper, $p$ denotes a prime number, $e(z)$ abbreviates $\mathrm e^{2\pi iz}$, and $\varphi(q)$ is Euler's totient function. The implicit constants in Vinogradov's and Landau's symbols will depend on the value assigned to $\varepsilon$. Let $n$ be a sufficiently large integer subject to $n\equiv1\pmod2$ and $n\equiv\pm1\pmod3$, and write
\begin{equation*}
P_k=n^{1/k},\qquad L=\log n.
\end{equation*}
For $k\geqslant2$, define
\begin{equation}\label{eq:prime-weyl-sums}
f_k(\alpha)=
\sum_{P_k/2<p\leqslant P_k}(\log p)e(\alpha p^k).
\end{equation}
For simplicity, we often write $f_k$ for $f_k(\alpha)$, and define
\begin{equation}\label{eq:generating-function}
F(\alpha)=f_2^2(\alpha)f_4^3(\alpha)f_5(\alpha)
f_8(\alpha).
\end{equation}
By orthogonality, the integral
\begin{equation}\label{eq:representation-integral}
\nu_1^*(n):=\int_0^1F(\alpha)e(-\alpha n)\,\mathrm d\alpha
\end{equation}
counts the number of solutions of the equation (\ref{eq:target-representation}) in prime
variables constrained into certain intervals and weighted by $(\log p_1)\dots(\log p_7)$.
Let
\begin{equation}\label{eq:main-term-exponent}
\Theta=\frac34+\frac15+\frac18=\frac{43}{40}.
\end{equation}
Our goal is to demonstrate that the lower bound
\begin{equation}\label{eq:weighted-lower-bound}
\nu(n)\gg n^\Theta
\end{equation}
holds for all sufficiently large integers $n$ subject to $n\equiv1\pmod2$ and $n\equiv\pm1\pmod3$. This is a quantitative form of the theorem.

For a fixed sufficiently large real number $B\geqslant1$, let $\mathfrak M$ be the union of all intervals
\begin{equation*}
\mathfrak M(q,a)=
\big\{\alpha \in [0,1] : |\alpha-a/q|\leqslant L^B/n\big\},
\end{equation*}
with $1\leqslant a\leqslant q$, $(a,q)=1$ and $1\leqslant q\leqslant L^B$. Let
$\mathfrak m=[0,1]\setminus\mathfrak M$. Theorem~\ref{thm:waring-goldbach-2244458}
will be proved once we establish the bounds
\begin{equation}\label{eq:circle-method-targets}
\int_{\mathfrak M}F(\alpha)e(-\alpha n)\,\mathrm d\alpha\gg n^\Theta
\quad \textrm{and}\quad
\int_{\mathfrak m}|F(\alpha)|\,\mathrm d\alpha\ll n^\Theta L^{-1}
\end{equation}
for every sufficiently large integer $n$ with $n\equiv1\!\pmod2$ and $n\equiv\pm1\!\pmod3$. The contribution from $\mathfrak M$
proceeds along a well trodden path. An outline is given below.

\section{Major arcs asymptotics}\label{sec:major-arcs}

For $k\geqslant2$, we write
\begin{equation*}
S_k(q,a)=\sum_{\substack{1\leqslant x\leqslant q\\(x,q)=1}}
e\bigg(\frac{ax^k}{q}\bigg),
\qquad
v_k(\beta)=\frac1k\sum_{2^{-k}n<m\leqslant n}m^{1/k-1}e(\beta m).
\end{equation*}
Proceeding in the manner of Hua's proof of Lemma 6 in \cite{Hua1938}, we arrive at the following result.

\begin{lemma}
\label{lem:major-arc-approximation}
Let $a$ and $q$ be integers satisfying $1\leqslant a\leqslant q$, $(a,q)=1$ and $1\leqslant q\leqslant L^B$. Then, for $k\geqslant2$ and for $\alpha=a/q+\beta\in\mathfrak M(q,a)$, there holds
\begin{equation*}
f_k(\alpha)=\frac{S_k(q,a)}{\phi(q)}v_k(\alpha - a/q)
+O\big(P_ke^{-c_1\sqrt L}\big),
\end{equation*}
where $c_1>0$ is a constant depending at most on $B$.
\end{lemma}

\begin{lemma}\label{lem:major-arc-lower-bound}
For every sufficiently large integer $n$ with $n\equiv1\!\pmod2$ and $n\equiv\pm1\!\pmod3$, one has
\begin{equation*}
\int_{\mathfrak M}F(\alpha)e(-\alpha n)\,\mathrm d\alpha
\gg n^\Theta.
\end{equation*}
\end{lemma}

\begin{proof}
We multiply the approximations for $f_k(\alpha)$ to estimate $F(\alpha)$. Define
\begin{equation*}
U(q,a)=S_2^2(q,a)S_4^3(q,a)S_5(q,a)S_8(q,a),
\qquad
w(\beta)=v_2^2(\beta)v_4^3(\beta)v_5(\beta)v_8(\beta).
\end{equation*}
By Lemma \ref{lem:major-arc-approximation}, we derive that
\begin{equation*}
F(\alpha)=(\varphi(q))^{-7}U(q,a)w(\beta)
+O(n^{\Theta+1}L^{-4B}).
\end{equation*}
We integrate this over $\mathfrak M$, a set of measure $O(L^{3B}/n)$, and obtain
\begin{equation}
\int_{\mathfrak M}F(\alpha)e(-\alpha n)\,\mathrm d\alpha
= \sum_{q\leqslant L^B}\frac{A_n(q)}{\varphi^7(q)}
\int_{-L^B/n}^{L^B/n}w(\beta)e(-\beta n)\,\mathrm d\beta
+ O(n^\Theta L^{-B}),
\label{eq:major-arc-reduction}
\end{equation}
where
\begin{equation*}
A_n(q)=\sum_{\substack{1\leqslant a\leqslant q\\(a,q)=1}}U(q,a)e(-an/q).
\end{equation*}
By Lemma 5 of Hua \cite{Hua1938}, we have $U(q,a)\ll q^{7/2+\varepsilon}$ for coprime $a,q$, and hence $A_n(q)\ll q^{9/2+\varepsilon}$ uniformly in $n$. It follows that the series
\begin{equation*}
\mathfrak S(n)=\sum_{q=1}^\infty\frac{A_n(q)}{\varphi^7(q)}
\end{equation*}
converges absolutely. Moreover, routine calculations show that
\begin{equation*}
\sum_{q>L^B}\frac{|A_n(q)|}{\varphi^7(q)}
\ll \sum_{q>L^B}q^{-5/2+\varepsilon}
\ll L^{-3B/2+\varepsilon},
\end{equation*}
which implies that
\begin{equation}\label{asymp-truncated}
\sum_{q\leqslant L^B}\frac{A_n(q)}{\varphi^7(q)}
=\mathfrak S(n)+O(L^{-3B/2+\varepsilon})
\end{equation}
holds uniformly in $n$. Also, by Lemma 6.2 of Vaughan \cite{Vaughan1997}, for $|\beta|\leqslant 1/2$, one has
\begin{equation}
w(\beta)\ll n^{\Theta+1}(1+n|\beta|)^{-2}.
\label{eq:major-kernel-bound}
\end{equation}
Thus, we can replace the sum in (\ref{eq:major-arc-reduction}) by $\mathfrak S(n)$ with negligible errors. We aim to prove that $\mathfrak S(n)\gg1$ for all large $n$ subject to $n\equiv1\!\pmod2$ and $n\equiv\pm1\!\pmod3$. Based on the arguments in the proof of Lemma 2.11 of Vaughan \cite{Vaughan1997}, we can likewise establish that $A_n(q)$ is multiplicative. Moreover, it follows from Lemma 4 of Hua \cite{Hua1938} that $A_n(q)=0$ unless $q$ is square--free.  This yields
\begin{equation}
\mathfrak S(n)=\prod_p\big(1+(p-1)^{-7}A_n(p)\big).
\label{eq:singular-series-product}
\end{equation}

Let $M_n(p)$ denote the number of solutions to
\begin{equation}
x_1^2+x_2^2+x_3^4+x_4^4+x_5^4+x_6^5+x_7^8=n
\label{eq:local-congruence}
\end{equation}
in the finite field $\mathbb F_p$, where all variables are non--zero. By orthogonality mod $p$, one has
\begin{equation*}
1+(p-1)^{-7}A_n(p)=p(p-1)^{-7}M_n(p).
\end{equation*}

For $p=2$, the only non--zero residue is $1$. Hence every term on the
left--hand side of (\ref{eq:local-congruence}) is congruent to $1$ modulo $2$, and therefore
\begin{equation*}
	x_1^2+x_2^2+x_3^4+x_4^4+x_5^4+x_6^5+x_7^8
	\equiv 7\equiv 1\pmod{2}.
\end{equation*}
Thus (\ref{eq:local-congruence}) is solvable modulo $2$ precisely when $n$ satisfies $n\equiv1\!\pmod2$.

For $p=3$, every non-zero residue $x$ satisfies
\begin{equation*}
x^2\equiv x^4\equiv x^8\equiv1\pmod3,
\qquad x^5\equiv x\pmod3.
\end{equation*}
Thus the six even powers contribute $6\equiv0\pmod3$, and the left--hand
side of (\ref{eq:local-congruence}) is congruent to $x_6\pmod3$.
Hence the six even powers contribute $6\equiv 0\pmod 3$,
while $x_6^5\equiv \pm1\pmod 3$.
Thus the left-hand side of (\ref{eq:local-congruence}) may be
\begin{equation*}
	6+1\equiv 1,\qquad 6-1\equiv 2\pmod 3.
\end{equation*}
Therefore (\ref{eq:local-congruence}) is solvable modulo $3$
precisely when $n$ satisfies $n\equiv\pm1\!\pmod3$.

For $p=5$, every non-zero residue $x$ satisfies
\begin{equation*}
	x^4\equiv x^8\equiv 1\pmod 5,
	\qquad x^5\equiv x\pmod 5.
\end{equation*}
Taking $x_2=x_3=x_4=x_5=x_7=1$, the left-hand side
of (\ref{eq:local-congruence}) is congruent to $x_1^2+x_6 \pmod 5$.
For any $r\in\mathbb F_5$ with $r\ne 1$,
take $x_1=1$ and $x_6=r-1$, so that
\begin{equation*}
	x_1^2+x_6\equiv 1+(r-1)\equiv r\pmod 5.
\end{equation*}
For $r=1$, take $x_1=x_6=2$, so that
\begin{equation*}
x_1^2+x_6\equiv 4+2\equiv 1\pmod 5.
\end{equation*}
Thus every residue class modulo $5$ is represented.

For $p=7$, for every non-zero
residue $x$, Fermat's theorem gives
\begin{equation*}
x^5\equiv x^{-1}\pmod 7.
\end{equation*}
Again taking
$x_2=x_3=x_4=x_5=x_7=1$, the left-hand side
of (\ref{eq:local-congruence}) is congruent to $5+x_1^2+x_6^5 \pmod 7$.
For any $r\in\mathbb F_7$ with $r\ne 6$,
take $x_1=1$ and $x_6=(r-6)^{-1}$, so that
\begin{equation*}
	5+x_1^2+x_6^5 \equiv 5+x_1^2+x_6^{-1} \equiv 6 + (r-6) \equiv r \pmod 7.
\end{equation*}
For $r=6$, take $x_1=x_6=2$, so that
\begin{equation*}
	5+x_1^2+x_6^5
	\equiv 5+4+4\equiv 6\pmod 7.
\end{equation*}
Hence every residue class modulo $7$ is represented.

It remains to consider $p\geqslant11$. Set
\begin{equation*}
\mathscr A_k(p)=\{x^k:x\in\mathbb F_p^\times\}.
\end{equation*}
Then $|\mathscr A_k(p)|=(p-1)/(k,p-1)$, and repeated application of the Cauchy--Davenport theorem gives
\begin{align*}
|\mathscr A_2+\mathscr A_2+\mathscr A_4+\mathscr A_4+\mathscr A_4+\mathscr A_5+\mathscr A_8|
\geqslant & \,\, \min\bigg\{p,
2\frac{p-1}{2}+3\frac{p-1}{4}
+\frac{p-1}{5}+\frac{p-1}{8}-6\bigg\}
\\
= & \,\, \min\bigg\{p,\frac{83}{40}(p-1)-6\bigg\}=p.
\end{align*}
Here the last equality follows from $p\geqslant11$. Hence (\ref{eq:local-congruence}) is solvable with all variables non-zero, and $M_n(p)\geqslant1$.

Hence the factors in (\ref{eq:singular-series-product}) are positive for every $n$ with $n\equiv1\!\pmod2$ and $n\equiv\pm1\!\pmod3$, and the uniform convergence with respect to $n$ gives a number $p_0$ such that
\begin{equation*}
\mathfrak S(n)\geqslant\frac12\prod_{p\leqslant p_0}
\big(1+(p-1)^{-7}A_n(p)\big)
\geqslant\frac12\prod_{p\leqslant p_0}p^{-6}.
\end{equation*}
Thus, the lower bound $\mathfrak S(n)\gg1$ holds for all large $n$ with $n\equiv1\!\pmod2$ and $n\equiv\pm1\!\pmod3$.

In order to conclude the proof, it suffices to show that the integral on the right-hand side of (\ref{eq:major-arc-reduction}) is $\gg n^\Theta$. Using the estimate (\ref{eq:major-kernel-bound}), we deduce that
\begin{equation}
\int_{-L^B/n}^{L^B/n}w(\beta)e(-\beta n)\,\mathrm d\beta
=  \int_{-1/2}^{1/2}w(\beta)e(-\beta n)\,\mathrm d\beta
+ O(n^\Theta L^{-B}).
\label{eq:singular-integral-extension}
\end{equation}
By orthogonality, the integral on the right-hand side of (\ref{eq:singular-integral-extension}) is
\begin{equation}
\frac1{2^2\cdot4^3\cdot5\cdot8}
\sum(m_1m_2)^{-1/2}(m_3m_4m_5)^{-3/4}m_6^{-4/5}m_7^{-7/8},
\label{eq:singular-integral-sum}
\end{equation}
where the summation runs over $m_1,m_2,m_3,m_4,m_5,m_6,m_7$ subject to
\begin{equation*}
m_1+m_2+m_3+m_4+m_5+m_6+m_7=n,
\end{equation*}
with
\begin{equation*}
2^{-2}n<m_1,m_2\leqslant n,
\qquad
2^{-4}n<m_3,m_4,m_5\leqslant n,
\end{equation*}
\begin{equation*}
2^{-5}n<m_6\leqslant n,
\qquad 2^{-8}n<m_7\leqslant n.
\end{equation*}
If we restrict the free variables to
\begin{equation*}
\frac14n<m_1,m_2\leqslant\frac{33}{128}n,
\qquad
\frac1{16}n<m_3,m_4,m_5\leqslant\frac9{128}n,
\qquad
\frac1{32}n<m_6\leqslant\frac5{128}n,
\end{equation*}
then
\begin{equation*}
m_1+\cdots+m_6\leqslant\frac{49}{64}n,
\end{equation*}
and thus
\begin{equation*}
m_7=n-(m_1+\cdots+m_6)\geqslant\frac{15}{64}n>\frac1{256}n,
\end{equation*}
and also $m_7<n$. Every such choice therefore produces an admissible $m_7$. There are $\gg n^6$ choices of the six free variables, and the weight in (\ref{eq:singular-integral-sum}) is $\asymp n^{-1}n^{-9/4}n^{-4/5}n^{-7/8}$. Consequently, we deduce that
\begin{equation*}
\int_{-1/2}^{1/2}w(\beta)e(-\beta n)\,\mathrm d\beta
\gg n^{6-1-9/4-4/5-7/8}=n^{43/40}=n^\Theta,
\end{equation*}
which combined with (\ref{eq:singular-integral-extension}), (\ref{eq:major-arc-reduction}), (\ref{asymp-truncated}) and the property $\mathfrak S(n)\gg1$ for $n\equiv1\!\pmod2$ and $n\equiv\pm1\!\pmod3$ completes the proof of Lemma \ref{lem:major-arc-lower-bound}.
\end{proof}

\section{Minor arcs estimate}\label{sec:minor-arcs}
In this section, we shall concentrate on dealing with the contribution from the minor arcs. First, we list several lemmas which are necessary to bound the integral.

\begin{lemma}\label{l26}
Let $f_k(\alpha)$ be defined as in (\ref{eq:prime-weyl-sums}). Suppose that $1 \leqslant j \leqslant k$. Then one has
\begin{equation*}
\int_{0}^{1} \big| f_k(\alpha) \big|^{2^{j}}  \mathrm{d}\alpha \ll P_k^{2^{j} - j + \varepsilon},
\end{equation*}
\end{lemma}
\begin{proof}
See Lemma 2.5 of Vaughan \cite{Vaughan1997} or Theorem 4 of Hua \cite{Hua-book}.
\end{proof}

\begin{lemma}\label{2-5-mean-value}
Let $f_k(\alpha)$ be defined as in (\ref{eq:prime-weyl-sums}). Then, we have
\begin{equation*}
\int_0^1\big|f_2^2(\alpha)f_5^6(\alpha)\big|\mathrm{d}\alpha\ll P_5^{6+\varepsilon}.
\end{equation*}
\end{lemma}
\begin{proof}
Trivially, the conclusion can be deduced by counting the number of solutions of the underlying Diophantine equation
\begin{equation*}
x_1^2-x_2^2=y_1^5+y_2^5+y_3^5-y_4^5-y_5^5-y_6^5
\end{equation*}
with $P_2/2\leqslant x_1,x_2\leqslant P_2$ and $P_5/2\leqslant y_1,\dots,y_6\leqslant P_5$. If $x_1\not= x_2$, the
contribution is bounded by $P_5^{6+\varepsilon}$. If $x_1=x_2$, by Cauchy's inequality and Lemma \ref{l26}, we know that the contribution is bounded by
\begin{align*}
\ll & \,\, P_2\cdot\int_0^1\big|f_5(\alpha)\big|^6\mathrm{d}\alpha
\ll  P_2\cdot\bigg(\int_0^1\big|f_5(\alpha)\big|^4\mathrm{d}\alpha\bigg)^{1/2}
           \bigg(\int_0^1\big|f_5(\alpha)\big|^8\mathrm{d}\alpha\bigg)^{1/2}
               \nonumber \\
\ll & \,\, P_2\cdot(P_5^{2+\varepsilon})^{1/2}\cdot(P_5^{5+\varepsilon})^{1/2}\ll P_5^{6+\varepsilon}.
\end{align*}
Combining the above two cases, we derive the desired conclusion, which completes the proof of Lemma \ref{2-5-mean-value}.
\end{proof}

\begin{lemma}\label{lem:large-prime-values}
Let $k\geqslant4$ and $X$ be sufficiently large. Define
$\sigma_k^{-1}=\min\{2^{k-1},k(k-1)\}$, and set
\begin{equation*}
g_k(\alpha;X)=\sum_{X<p\leqslant2X}e(\alpha p^k).
\end{equation*}
Suppose that $X^{2\sigma_k/3}\leqslant Z\leqslant X^{9/20}$. Then, for every sufficiently small $\varepsilon>0$, either one has the bound
\begin{equation*}
g_k(\alpha;X)\ll X^{1-\sigma_k/3+\varepsilon},
\end{equation*}
or else there exist integers $a$ and $q$ such that
\begin{equation*}
1\leqslant q\leqslant Z, \qquad (a,q)=1, \qquad |q\alpha-a|\leqslant ZX^{-k},
\end{equation*}
in which case
\begin{equation*}
g_k(\alpha;X)
\ll \frac{X^{1+\varepsilon}}
{(q+X^k|q\alpha-a|)^{1/2}}.
\end{equation*}
\end{lemma}
\begin{proof}
See Lemma 2.2 of Kumchev and Wooley \cite{KumchevWooley2017}.
\end{proof}

\begin{lemma}
\label{lem:kumchev-prime-sum}
Suppose that $P$ is sufficiently large, and define
\begin{equation*}
G_k(\alpha;P)=\sum_{P\leqslant p<2P}e(\alpha p^k).
\end{equation*}
Let $k\in\mathbb N$ and $\alpha\in\mathbb R$, and suppose that there exist $a\in\mathbb{Z}$ and $q\in\mathbb{N}$
satisfying
\begin{equation*}
1\leqslant q\leqslant Q, \qquad (a,q)=1, \qquad |q\alpha-a|<QP^{-k}
\end{equation*}
with $Q\leqslant P$. Then, for any fixed $\varepsilon>0$, one has
\begin{equation*}
G_k(\alpha;P)
\ll Q^{1/2}P^{11/20+\varepsilon}
+\frac{q^\varepsilon P(\log P)^c}
{(q+P^k|q\alpha-a|)^{1/2}},
\end{equation*}
where $c>0$ is an absolute constant and where the constant implied by ``$\ll$'' depends
at most on $k$ and $\varepsilon$.
\end{lemma}
\begin{proof}
See Theorem 2 of Kumchev \cite{Kumchev2006}.
\end{proof}

\begin{lemma}\label{Brudern-lemma}
For $f_k(\alpha)$ defined as in (\ref{eq:prime-weyl-sums}), let $2\leqslant k_1\leqslant k_2\leqslant\dots\leqslant k_s$ be natural numbers such that
\begin{equation*}
   \sum_{i=j+1}^s\frac{1}{k_i}\leqslant \frac{1}{k_j},\quad 1\leqslant j\leqslant s-1.
\end{equation*}
Then we have
\begin{equation*}
   \int_0^1\bigg|\prod_{i=1}^s f_{k_i}(\alpha)\bigg|^2\mathrm{d}\alpha\ll n^{\frac{1}{k_1}+\cdots+\frac{1}{k_s}+\varepsilon}.
\end{equation*}
\end{lemma}
\begin{proof}
  See Lemma 1 of Br\"{u}dern \cite{Brudern1987I}.
\end{proof}

Now, we shall deal with the contribution of the integral over minor arcs to
(\ref{eq:representation-integral}) which arises from the set
\begin{equation}\label{eq:small-fourth-set}
\mathfrak E=\big\{\alpha\in[0,1]:|f_4(\alpha)|\leqslant P_4^t\big\}
\end{equation}
with
\begin{equation*}
t=\frac{31}{32}.
\end{equation*}

\begin{lemma}\label{lem:small-fourth-values}
There exists a fixed $\eta>0$ such that
\begin{equation}\label{eq:admissible-t}
\int_{\mathfrak E}|F(\alpha)|\,\mathrm d\alpha
\ll n^{\Theta-\eta}.
\end{equation}
\end{lemma}
\begin{proof}
By H\"older's inequality, we deduce that
\begin{equation*}
\int_{\mathfrak E}|F(\alpha)|\,\mathrm d\alpha
\leqslant J_1^{1/6}J_2^{7/12}J_3^{1/4}
\sup_{\alpha\in\mathfrak E}|f_4(\alpha)|^{1/6},
\end{equation*}
where
\begin{equation}\label{J-1-2-3-integral}
J_1=\int_0^1\big|f_2^2f_5^6\big|\mathrm d\alpha,
\qquad
J_2=\int_0^1\big|f_2f_4^2\big|^2\mathrm d\alpha,\qquad
J_3=\int_0^1\big|f_2f_4f_8^2\big|^2\mathrm d\alpha.
\end{equation}
It follows from Lemma \ref{2-5-mean-value} and Lemma \ref{Brudern-lemma} that
\begin{equation}\label{eq:square-mean-value}
J_1\ll n^{6/5+\varepsilon},\qquad J_2\ll n^{1+\varepsilon},\qquad J_3\ll n^{1+\varepsilon},
\end{equation}
which combined with (\ref{eq:small-fourth-set}) yields that
\begin{equation*}
\int_{\mathfrak E}|F(\alpha)|\,\mathrm d\alpha
\ll  n^{\frac16\cdot\frac65+\frac7{12}+\frac14+\frac{t}{24}+\varepsilon}
=  n^{31/30+t/24+\varepsilon}.
\end{equation*}
Since $\Theta=43/40$, the strict upper bound condition in (\ref{eq:admissible-t}) is precisely derived from
\begin{equation}\label{eq:small-value-closure-margin}
\Theta-\bigg(\frac{31}{30}+\frac{t}{24}\bigg)
=\frac{1-t}{24}
=\frac1{768}>0.
\end{equation}
By choosing the $\varepsilon$ in the above mean value estimate with $0<\varepsilon<1/1536$ and putting
$\eta=1/768-\varepsilon$, it is easy to see that $\eta>0$, and the desired estimate follows.
\end{proof}

In order to treat the complement of $\mathfrak E$, we apply
Lemma \ref{lem:large-prime-values}. For $k=4$, one has
$\sigma_4=1/8$ and
\begin{equation*}
1-\frac{\sigma_4}{3}=\frac{23}{24}<t.
\end{equation*}
Taking $Q=n^{1/32}$, $X=P_4/2$ and $Z=Q/16$, we have, for
sufficiently large $n$, one has
\begin{equation*}
X^{2\sigma_4/3}=X^{1/12}\leqslant Z\leqslant X^{9/20}.
\end{equation*}
which satisfies the condition of Lemma \ref{lem:large-prime-values}. Suppose that
$\alpha\notin\mathfrak E$, and choose $\varepsilon>0$ such that
$23/24+2\varepsilon<t$. By partial summation, the first alternative case of Lemma \ref{lem:large-prime-values} would give
\begin{equation*}
|f_4(\alpha)|\ll P_4^{23/24+2\varepsilon}<P_4^t,
\end{equation*}
which is impossible. Hence the second alternative case holds.
Since $ZX^{-4}=Q/n$, there exist integers $a,q$ satisfying
\begin{equation*}
1\leqslant q\leqslant Z\leqslant Q,\qquad
(a,q)=1,\qquad |q\alpha-a|\leqslant Q/n.
\end{equation*}
It follows that $[0,1]\setminus\mathfrak E\subseteq\mathfrak R$, where
$\mathfrak R$ is the union of the intervals
\begin{equation}
\mathfrak R(q,a)=
\big\{\alpha \in [0,1] : |q\alpha-a|\leqslant Q/n \big\},
\label{eq:rational-arcs}
\end{equation}
with $1\leqslant a\leqslant q$, $(a,q)=1$ and $1\leqslant q\leqslant Q$. For sufficiently large $n$, the intervals $\mathfrak R(q,a)$ occurring are literally pairwise disjoint.

Define a function $\Upsilon:\mathfrak R\to[0,1]$ by
\begin{equation*}
	\Upsilon(\alpha)=(q+n|q\alpha-a|)^{-1},
	\qquad \alpha\in\mathfrak R(q,a).
\end{equation*}
For $\alpha\in\mathfrak R(q,a)$, one has $q\leqslant Q\leqslant P_k/2$ and
$|q\alpha-a|\leqslant QP_k^{-k}<Q(P_k/2)^{-k}$ for $k=2,4,8$.
Applying Lemma \ref{lem:kumchev-prime-sum} with $P=P_k/2$, and then by using partial summation and
$(P_k/2)^k\asymp n$, we obtain
\begin{align}\label{eq:prime-sum-pruning-bound}
f_k(\alpha)
\ll & \,\, Q^{1/2}P_k^{11/20+\varepsilon}L
+q^\varepsilon P_kL^C\Upsilon(\alpha)^{1/2}
\nonumber\\
\ll & \,\, P_kL^C\Upsilon(\alpha)^{1/2-\varepsilon},
\end{align}
where $C>0$ is a constant.
Hence, by using the trivial bound for $f_5$, we deduce that
\begin{equation*}
|F(\alpha)|
\ll P_2^2P_4^3P_5P_8L^{C_1}\Upsilon(\alpha)^{3-6\varepsilon}
\ll n^{\Theta+1}L^{C_2}\Upsilon(\alpha)^{5/2},
\end{equation*}
where the last estimate follows by taking the same $\varepsilon$ sufficiently
small subject to $6\varepsilon\leqslant1/2$.

For $\alpha\in\mathfrak{m}\backslash\mathfrak{E}\subseteq\mathfrak{R}\backslash\mathfrak{M}$, we set $\alpha=a/q+\beta$. Then, either one has
\begin{equation}\label{complement-case-1}
L^B<q\leqslant Q, \qquad |\beta|\leqslant\frac{Q}{qn},
\end{equation}
or else
\begin{equation}\label{complement-case-2}
q\leqslant L^B, \qquad \frac{L^B}{n}<|\beta|\leqslant\frac{Q}{qn}.
\end{equation}
If (\ref{complement-case-1}) holds, then the contribution to $\mathfrak{m}$ is
\begin{align*}
 \ll & \,\, \sum_{L^B<q\leqslant Q}\sum_{\substack{a=1\\(a,q)=1}}^q\int_{\mathfrak{R}(q,a)}
            |F(\alpha)|\mathrm{d}\alpha
                \nonumber \\
 \ll & \,\, n^{\Theta+1}L^{C_2}\sum_{L^B<q\leqslant Q}\sum_{\substack{a=1\\(a,q)=1}}^q
            \int_{-\frac{Q}{qn}}^{\frac{Q}{qn}}\frac{\mathrm{d}\beta}{(q+qn|\beta|)^{5/2}}
                 \nonumber \\
 \ll & \,\, n^{\Theta+1}L^{C_2}\sum_{L^B<q\leqslant Q}\sum_{\substack{a=1\\(a,q)=1}}^q\frac{1}{qn}
            \int_0^Q\frac{\mathrm{d}u}{(q+u)^{5/2}}
                 \nonumber \\
 \ll & \,\, n^\Theta L^{C_2}\sum_{L^B<q\leqslant Q}q^{-3/2}
            \ll n^\Theta L^{C_2-B/2}\ll n^\Theta L^{-1}.
\end{align*}
If (\ref{complement-case-2}) holds, then the contribution to $\mathfrak{m}$ is
\begin{align*}
 \ll & \,\, \sum_{q\leqslant L^B}
 			\sum_{\substack{0\leqslant a\leqslant q\\(a,q)=1}}
            \int_{\mathfrak{R}(q,a)\backslash\mathfrak{M}(q,a)}|F(\alpha)|\mathrm{d}\alpha
                \nonumber \\
 \ll & \,\, n^{\Theta+1}L^{C_2}\sum_{q\leqslant L^B}
 			\sum_{\substack{0\leqslant a\leqslant q\\(a,q)=1}}
            \int_{\frac{L^B}{n}<|\beta|\leqslant\frac{Q}{qn}}\frac{\mathrm{d}\beta}{(q+qn|\beta|)^{5/2}}
                 \nonumber \\
 \ll & \,\, n^{\Theta+1}L^{C_2}\sum_{q\leqslant L^B}
 			\sum_{\substack{0\leqslant a\leqslant q\\(a,q)=1}}\frac{1}{qn}
            \int_{qL^B}^Q\frac{\mathrm{d}u}{u^{5/2}}
                 \nonumber \\
 \ll & \,\, n^{\Theta+1}L^{C_2}\sum_{q\leqslant L^B}
 			\sum_{\substack{0\leqslant a\leqslant q\\(a,q)=1}}
            \frac{1}{q^{5/2}L^{3B/2}n}
                 \nonumber \\
 \ll & \,\, n^\Theta L^{C_2-3B/2}\sum_{q\leqslant L^B}q^{-3/2}\ll n^\Theta L^{C_2-3B/2}\ll n^\Theta L^{-1}.
\end{align*}
Taking $B>2C_2+2$ and combining the above two cases, we deduce that
\begin{equation}\label{minor-RM-upper}
\int_{\mathfrak R\setminus\mathfrak M}|F(\alpha)|\,\mathrm d\alpha
\ll n^\Theta L^{-1}.
\end{equation}
Accordingly, by Lemma \ref{lem:small-fourth-values} and  (\ref{minor-RM-upper}), we finally derive that
\begin{align}\label{minor-final-upper}
           \int_{\mathfrak m}|F(\alpha)|\,\mathrm d\alpha
\ll & \,\, \int_{\mathfrak m\cap\mathfrak{E}}|F(\alpha)|\mathrm d\alpha
           +\int_{\mathfrak m\backslash\mathfrak{E}}|F(\alpha)|\mathrm d\alpha
                 \nonumber \\
\ll & \,\, n^{\Theta-\eta}+n^\Theta L^{-1}
\ll  n^\Theta L^{-1}.
\end{align}
Based on Lemma \ref{lem:major-arc-lower-bound} and (\ref{minor-final-upper}), we know that the two bounds in
(\ref{eq:circle-method-targets}) hold. Since each representation contributing to $\nu_1^*(n)$ carries a logarithmic weight $L^7$, then we finally derive that
\begin{equation*}
	\nu_1(n)\asymp\nu_1^*(n)L^{-7}\gg n^\Theta L^{-7}
	\gg n^{43/40}(\log n)^{-7}.
\end{equation*}
This completes the proof of Theorem \ref{thm:waring-goldbach-2244458}.

\section{Proof of Theorem \ref{thm:waring-goldbach-2234568}}\label{sec:proof-remark}
Let $n$ be a sufficiently large odd integer, and define $f_3,f_6$
by (\ref{eq:prime-weyl-sums}). Set
\begin{equation*}
\widetilde F(\alpha)
=f_2^2(\alpha)f_3(\alpha)f_4(\alpha)f_5(\alpha)f_6(\alpha)f_8(\alpha).
\end{equation*}
By orthogonality, the integral
\begin{equation*}
\nu_2^*(n):=\int_0^1\widetilde F(\alpha)e(-\alpha n)\mathrm d\alpha
\end{equation*}
counts the representations in (\ref{eq:target-representation-2234568})
with the same logarithmic weights and with each prime variable in its
corresponding interval $(P_k/2,P_k]$. For the major arcs, define $S_k(q,a)$ and $v_k(\beta)$ as in Section \ref{sec:major-arcs}, and put
\begin{align*}
\widetilde U(q,a)= & \,\,S_2^2(q,a)S_3(q,a)S_4(q,a)S_5(q,a)S_6(q,a)S_8(q,a),\\
\widetilde w(\beta)= & \,\,v_2^2(\beta)v_3(\beta)v_4(\beta)v_5(\beta)v_6(\beta)v_8(\beta).
\end{align*}
The approximation in Lemma \ref{lem:major-arc-approximation} also holds
in this situation. Thus, by the same arguments as in Lemma \ref{lem:major-arc-lower-bound},
we obtain
\begin{equation*}
\int_{\mathfrak M}\widetilde F(\alpha)e(-\alpha n)\,\mathrm d\alpha
=\widetilde{\mathfrak S}(n)
\int_{-1/2}^{1/2}\widetilde w(\beta)e(-\beta n)\,\mathrm d\beta
+O(n^\Xi L^{-B}),
\end{equation*}
where $\widetilde{\mathfrak S}(n)$ is the singular series obtained by
replacing $U(q,a)$ with $\widetilde U(q,a)$ in Section \ref{sec:major-arcs}, and
\begin{equation*}
\Xi=\frac{1}{3}+\frac{1}{4}+\frac{1}{5}+\frac{1}{6}+\frac{1}{8}=\frac{43}{40}.
\end{equation*}
As in Section \ref{sec:major-arcs}, the singular series
$\widetilde{\mathfrak S}(n)$ converges absolutely
and admits an Euler product over primes. Moreover, we can verify that, for $n\equiv1\!\pmod2$, the congruence
\begin{equation*}
x_1^2+x_2^2+x_3^3+x_4^4+x_5^5+x_6^6+x_7^8\equiv n\pmod p
\end{equation*}
is solvable with all variables non-zero for every prime $p$.
For $p=2$, this follows from $n$ being odd.
For $p=3,5,7$, a direct check shows that every
residue class modulo $p$ is represented.
For $p\geqslant 11$, the Cauchy--Davenport theorem gives
\begin{equation*}
|\mathscr A_2+\mathscr A_2+\mathscr A_3+
\mathscr A_4+\mathscr A_5+\mathscr A_6+\mathscr A_8|
\geqslant
\min\left\{p,\frac{83}{40}(p-1)-6\right\}=p.
\end{equation*}
Consequently, as is shown in Section \ref{sec:major-arcs}, we obtain $\widetilde{\mathfrak S}(n)\gg 1$ for
all odd $n$.

In order to bound the singular integral below, we write
$(k_1,k_2,\dots,k_7)=(2,2,3,4,5,6,8)$. By orthogonality, one has
\begin{equation*}
\int_{-1/2}^{1/2}\widetilde w(\beta)e(-\beta n)\,\mathrm d\beta
=\frac1{2^2\cdot3\cdot4\cdot5\cdot6\cdot8}
\sum_{\substack{m_1+\cdots+m_7=n\\
2^{-k_j}n<m_j\leqslant n\\ 1\leqslant j\leqslant7}}
\prod_{j=1}^7m_j^{1/k_j-1}.
\end{equation*}
If one constrains the first six variables into the range
\begin{equation*}
2^{-k_j}n<m_j\leqslant\bigg(2^{-k_j}+\frac1{128}\bigg)n,
\qquad(1\leqslant j\leqslant6),
\end{equation*}
then $m_1+\cdots+m_6\leqslant25n/32$, and thus $7n/32\leqslant m_7<n$. These $\gg n^6$ choices are admissible,
and each weight is $\asymp n^{83/40-7}$. Hence
\begin{equation}\label{major-lower}
\int_{\mathfrak M}\widetilde F(\alpha)e(-\alpha n)\,\mathrm d\alpha
\gg n^{6+83/40-7}=n^\Xi.
\end{equation}
For the minor arcs, we retain $J_1,J_2,J_3$ defined as in (\ref{J-1-2-3-integral}) and introduce the integral
\begin{equation*}
K=\int_0^1\big|f_2f_3f_6\big|^2\mathrm d\alpha.
\end{equation*}
By Lemma \ref{Brudern-lemma}, one has $K\ll n^{1+\varepsilon}$,
which combined with (\ref{eq:square-mean-value}) and H\"older's inequality yields that
\begin{align}\label{C-F-upper}
\int_{\mathfrak E}\big|\widetilde F(\alpha)\big|\mathrm d\alpha
\leqslant & \,\,J_1^{1/6}J_2^{1/12}J_3^{1/4}K^{1/2}
\sup_{\alpha\in\mathfrak E}|f_4(\alpha)|^{1/6} \nonumber \\
\ll & \,\,n^{\frac16\cdot\frac65+\frac1{12}+\frac14+\frac12+\frac{t}{24}+\varepsilon}
=n^{\Xi-1/768+\varepsilon},
\end{align}
where the set $\mathfrak E$ and $t=31/32$ are defined exactly the same as in
(\ref{eq:small-fourth-set}). Taking $0<\varepsilon<1/1536$, one has
a fixed power saving.

Furthermore, we define the set $\mathfrak R$ as in (\ref{eq:rational-arcs}). By the arguments exactly the same as those in Section \ref{sec:minor-arcs}, we can use Lemma \ref{lem:kumchev-prime-sum} and partial summation to give
(\ref{eq:prime-sum-pruning-bound}) for $k=2,3,4,6,8$, which combined with the trivial bound for $f_5$ yields that
\begin{align*}
|\widetilde F(\alpha)|
\ll & \,\,P_2^2P_3P_4P_5P_6P_8L^{C_1}
\Upsilon(\alpha)^{3-6\varepsilon}\\
\ll & \,\,n^{\Theta+1}L^{C_2}\Upsilon(\alpha)^{5/2}.
\end{align*}
For $\alpha\in\mathfrak{m}\backslash\mathfrak{E}\subseteq\mathfrak{R}\backslash\mathfrak{M}$,
the conditions in (\ref{complement-case-1}) or
(\ref{complement-case-2}) therefore lead to
\begin{equation}\label{R-M-upper}
\int_{\mathfrak R\setminus\mathfrak M}
\big|\widetilde F(\alpha)\big|\mathrm d\alpha\ll n^\Xi L^{-1}.
\end{equation}
Since $\mathfrak{m}=(\mathfrak{m}\cap\mathfrak{E}) \cup(\mathfrak{m}\setminus\mathfrak{E})$,
it follows from (\ref{C-F-upper}) and  (\ref{R-M-upper}) that
\begin{equation*}
\int_{\mathfrak m}\big|\widetilde F(\alpha)\big|\mathrm d\alpha \ll n^\Xi L^{-1},
\end{equation*}
which combined with (\ref{major-lower}) yields that
\begin{equation*}
\int_0^1\widetilde F(\alpha)e(-\alpha n)\,\mathrm d\alpha
\gg n^\Xi.
\end{equation*}
Since each representation contributing to $\nu_2^*(n)$ carries a logarithmic weight $L^7$, then we finally derive that
\begin{equation*}
	\nu_2(n)\asymp\nu_2^*(n)L^{-7}\gg n^\Xi L^{-7}
	\gg n^{43/40}(\log n)^{-7}.
\end{equation*}
This completes the proof of Theorem \ref{thm:waring-goldbach-2234568}.

\section*{Acknowledgements}

%\noindent
The authors would like to appreciate the referee for his/her patience in refereeing this paper.
This work is supported by Beijing Natural Science Foundation (Grant No. 1242003), and the
National Natural Science Foundation of China (Grant Nos. 12471009, 12301006, 11901566, 12001047).

\end{document}